\documentclass[12pt]{article}
\usepackage[T1]{fontenc}
\usepackage{amsmath, amsthm, amssymb}
\usepackage[ansinew]{inputenc}
\usepackage{amsfonts}
\usepackage{amssymb}
\usepackage{xcolor}
\usepackage{enumitem}
\usepackage{amsthm}
\usepackage{dsfont}
\usepackage{tikz-cd}
\DeclareMathAlphabet{\mathbbmsl}{U}{bbm}{m}{sl}

\def\x{{\bf{x}}}

\def\c{\mathcal{C}}

\newtheorem{theorem}{Theorem}[section]

\newtheorem{example}[theorem]{Example}

\newtheorem{remark}[theorem]{Remark}
\newtheorem{lemma}[theorem]{Lemma}
\newtheorem{definition}[theorem]{Definition}
\newtheorem{proposition}[theorem]{Proposition}
\usepackage[numbers,sort&compress]{natbib}
\bibpunct[, ]{[}{]}{,}{n}{,}{,}

\begin{document}
\title{\large \textbf{On the Finiteness Property of the Polynomial Complementarity Problem}}
\author{Sonali Sharma$^{a,1}$, V. Vetrivel$^{a,2}$\\
{\small$^{a}$Department of Mathematics, Indian Institute of Technology Madras, Chennai, India}\\
{\small $^{1}$Email: ma24r005@smail.iitm.ac.in}\\
{\small $^{2}$Email: vetri@iitm.ac.in}\\}
\date{}
\maketitle

\begin{abstract}
This paper explores the finiteness of the solution set of the polynomial complementarity problem (PCP). To achieve this goal, we introduce two new classes of structured tensor tuples, namely the nondegenerate tensor tuple and the strong nondegenerate tensor tuple, as a generalization of nondegenerate tensors, and discuss their properties and interconnections. We investigate the finiteness of the solution set  of the PCP  in the context of these structured tensor tuples and establish a sufficient condition that guarantees a finite solution set. As a consequence, we establish a result related to the finiteness of the solution set of tensor complementarity problems. 
\end{abstract}

\section{Introduction}
For a given  (nonlinear) map $F: \mathbb{R}^n \to \mathbb{R}^n$ and a vector ${\bf q} \in \mathbb{R}^{n}$, the  nonlinear complementarity problem (NCP) seeks a vector $\mathbf{x} \in \mathbb{R}^n$ satisfying
\begin{equation}\label{NCP}
\mathbf{x} \geq {\bf 0},~ F(\mathbf{x}) + {\bf q} \geq {\bf 0},~ \text{and}~\mathbf{x}^{T}(F(\mathbf{x})+{\bf q})  = 0.
\end{equation}
This problem reduces to the linear complementarity problem (LCP), when the map in Eq. (\ref{NCP}) is a linear function given by ${\bf A}{\bf x}$, for a real square matrix ${\bf A}$ of order $n$. The nonlinear and linear complementarity problems (NCP and LCP) have been widely studied in the optimization literature due to their rich theory, algorithmic development, and wide-ranging applications in areas such as economic equilibrium, mathematical programming, traffic flow, and contact mechanics; see \cite{MR3396730,MR1955649,MR3896653} and references therein.
 
In Eq. (\ref{NCP}), if $F({\bf x}) = {\mathcal A}{\bf x}^{m-1}$, where $\mathcal{A}$ is a real tensor of order $m$ and dimension $n$ with entries as 
$$\mathcal{A} = (a_{i_{1}i_{2}...i_{m}}),~\text{for all}~1 \leq i_{1},i_{2},...i_{m} \leq n,$$
then the NCP reduces to the tensor complementarity problem (TCP). Here $\mathcal{A}{\bf x}^{m-1}$ is a homogeneous polynomial of degree $m-1$ and its $i$-th component is given as
$$(\mathcal{A}{\bf x}^{m-1})_{i} = \displaystyle{\sum_{i_{2},...,i_{m} =1}^n}{a_{ii_{2}...i_{m}}x_{i_{2}}...x_{i_{m}}}.$$
As a natural generalization of the LCP and a special case of the NCP, the TCP has been the subject of extensive research, especially in the last decade, with significant progress regarding the properties of the solution set, such as existence, uniqueness, finiteness \cite{MR4310678,MR3513266,MR3341670,sharma2023criterion}, solution methods and applications \cite{MR3613565,MR3989294,MR4023437,MR3998357,MR3513267}. Many generalizations of the TCP and their applications have also been studied in the literature, see, for instance \cite{jia2024,li2025extended,MR4576574,Yadav30082024,LI2026116873}.  

In recent years, the polynomial complementarity problem (PCP), which subsumes the  TCP  as a special case, has drawn increasing attention in the optimization community, after it was introduced by Gowda \cite{Gowda}. The PCP is  defined as finding a vector ${\bf x} \in \mathbb{R}^{n}$ such that
\begin{equation*}\label{PCP11}
\mathbf{x} \geq {\bf 0},~ \sum_{k=1}^{m-1} \mathcal{A}_k \mathbf{x}^{m-k} + \mathbf{q} \geq {\bf 0},~\text{and} ~\mathbf{x}^{T}(\sum_{k=1}^{m-1} \mathcal{A}_k \mathbf{x}^{m-k}+{\bf q}) = 0,
\end{equation*}
where each term $\mathcal{A}_{i}{\bf x}^{m-i}$ is a homogeneous polynomial of degree $m-i$, and corresponds to a real tensor $\mathcal{A}_{i}$ of order $m-(i-1)$ and dimension $n$. It has been shown that various classes of structured tensor tuples play an important role in the study of the properties of the solution set of the PCP. For instance, $Q$-tensor tuple \cite{shang2025q} provides existence of a solution, $S$-tensor tuple \cite{li2024strict} has been studied with regards to the strict feasibility of the PCP, the uniqueness of the solution has been studied with the help of semi-positive tensor tuple and copositive tensor tuple \cite{shang2023structured}. Moreover, the error bounds \cite{ling2018error}, bounds of the solution set \cite{xu2024bounds,li2024lower} and some generalizations and applications of the PCP \cite{hieu2020notes,shang2023mixed,zheng2020nonemptiness} have also been studied  in the literature. In particular, Hieu et al. \cite{hieu2020notes}  established a connection between the solution set of the PCP and a corresponding polynomial optimization problem, which provides an application of the PCP in optimization.

Among other properties of the solution set of these models, the finiteness has gained a lot of interest, particularly after the introduction of non-degenerate tensors to study the finiteness of the solution set of the TCP \cite{MR4310678}. Several other models of complementarity problems, such as the horizontal TCP, extended vertical TCP, extended horizontal TCP has been introduced and the finiteness of the corresponding solution sets has been analysed by generalizing the concept  of non-degenerate tensors. See, for example, \cite{li2025extended,LI2026116873,sharma2025extended}.

Inspired by the aforementioned works, a natural question arises: can the finiteness results for the solution set be extended to the PCP? Specifically, by introducing new classes of structured tensor tuples, can we analyse the finiteness of the solution set of the PCP? To explore this direction, we propose the notions of non-degenerate and strong non-degenerate tensor tuples (see Definitions \ref{non-degenerate tensor tuple} and \ref{strong non-degenerate tensor tuple}), as a generalization of non-degenerate tensors. With the help of these structured tensor tuples, we investigate the finiteness of the solution set of the PCP (see Section \ref{Section 4}), and provide a sufficient condition for the finiteness of the solution set of the PCP (refer to Theorem \ref{finiteness}). Consequently,  we also derive an equivalence between the finiteness of the solution set of the TCP and non-degenerate tensors,  for the class of matrix based tensors (see Theorem \ref{matrix based 1}), which is different from the class of row-diagonal tensors (see \cite[Theorem 3.1]{MR4310678}). 

The outline of this paper is as follows: In Section \ref{Section 2}, we recall some definitions and results, which will be useful in the sequel. In Section \ref{Section 3}, we introduce the non-degenerate tensor tuple and strong non-degenerate tensor tuple, and discuss their interconnections with non-degenerate tensors and some properties of these tensor tuples. Section \ref{Section 4} explores the finiteness of the solution set of the PCP with regards to these structured tensor tuples, and Section \ref{Section 5} establishes an equivalence between non-degenerate tensors and the finiteness of the solution set of the TCP, for the class of matrix based tensors. The conclusions are drawn in Section \ref{Section 6}.  

\section{Preliminaries and Notation}\label{Section 2}
The following notation will be used throughout the paper.
\begin{enumerate}
\item The $n$-dimensional Euclidean space with the usual inner product is denoted by $\mathbb{R}^{n},$ and elements of $\mathbb{R}^{n}$ are denoted by small bold letters, such as ${\bf x},{\bf y}$.
\item  For any natural number $n$, $[n]$ denotes the set $\{1,2,...,n\}.$
\item For any ${\bf x}, {\bf y}$ in $\mathbb{R}^{n},$ $({\bf x}*{\bf y})_{i}= (x_{i}y_{i}),$ for all $i \in [n]$. The bold vector ${\bf 0}$ denotes a vector in $\mathbb{R}^{n}$ having each of its entries equal to zero. For any ${\bf x} \in \mathbb{R}^{n}$ and a natural number $k$, ${\bf x}^{[k]}$ is a vector in $\mathbb{R}^{n}$ such that $({\bf x}^{[k]})_{i} = x_{i}^{k},$ for all $i \in [n].$
\item The set of all real square matrices of order $n$ is denoted by $\mathbb{R}^{n \times n}$, and the elements of $\mathbb{R}^{n \times n}$ are denoted as ${\bf A}, {\bf M},$ etc.  
\item  The set $\mathbb{T}(m,n)$ denotes the set of all real tensors of order $m$ and dimension $n$, and the elements of $\mathbb{T}(m,n)$ are denoted by math calligraphic letters, such as $\mathcal{A}_{1}, \mathcal{A}_{2},..$ and ${\mathcal I}$ denotes  the identity tensor.
\item A tensor is said to be a zero tensor if all of its entries are equal to zero.
\item  The set $\mathbb{T}(m,n) \times \mathbb{T}(m-1,n)\times...\times \mathbb{T}(2,n)$ is denoted by ${\Lambda(m,n)}.$
\end{enumerate}
In the following, we recall some definitions and results that will be useful in the sequel.

\begin{definition}\rm \label{LCP}\cite{MR3396730}
Let ${\bf A} \in \mathbb{R}^{n \times n}$ and ${\bf q} \in \mathbb{R}^{n}$. The linear complementarity problem is denoted by $\rm{LCP}({\bf A},{\bf q})$, and it is defined as finding a vector ${\bf x} \in \mathbb{R}^{n}$ such that
\begin{equation}\label{LCP112}
{\bf x} \geq {\bf 0},~{\bf A}{\bf x}+{\bf q} \geq {\bf 0}~\text{and}~{\bf x}^{T}({\bf A}{\bf x}+{\bf q}) =0.
\end{equation}
The collection of all the vectors ${\bf x}$ satisfying Eq. (\ref{LCP112}) is said to be the solution set of the $\rm{LCP}({\bf A},{\bf q})$  and is denoted by $\rm{SOL}({\bf A},{\bf q})$. 
\end{definition}

\begin{definition}\rm \cite[Definition 3.6.1]{MR3396730}
A matrix ${\bf A} \in  \mathbb{R}^{n \times n}$ is said to be a non-degenerate matrix if all of its principal minors are nonzero.
\end{definition}

\begin{definition}\rm \cite{MR3341670}
Let $\mathcal{A} \in \mathbb{T}(m,n)$ and ${\bf q} \in \mathbb{R}^{n}$. The tensor complementarity problem is denoted by $\rm{TCP}(\mathcal{A},{\bf q})$, and defined as finding a vector ${\bf x} \in \mathbb{R}^{n}$ such that
\begin{equation}\label{TCP}
{\bf x} \geq {\bf 0},~\mathcal{A}{\bf x}^{m-1}+{\bf q} \geq {\bf 0}~\text{and}~{\bf x}^{T}(\mathcal{A}{\bf x}^{m-1}+{\bf q}) =0.
\end{equation}
The collection of all the vectors such that Eq. (\ref{TCP}) is satisfied,  is said to be the solution set of the $\rm{TCP}(\mathcal{A},{\bf q})$, and denoted by $\rm{SOL}(\mathcal{A},{\bf q}).$
\end{definition}

\begin{definition}\rm\label{ND Tensor}\cite[Definition 2.1]{MR4310678}
A tensor $\mathcal{A} \in \mathbb{T}(m,n)$ is said to be
\begin{enumerate}
\item a non-degenerate tensor if ${\bf x} * \mathcal{A}{\bf x}^{m-1} ={\bf 0} \implies {\bf x} = {\bf 0}.$
\item  a strong non-degenerate tensor if $({\bf x}-{\bf y}) * (\mathcal{A}{\bf x}^{m-1}-\mathcal{A}{\bf y}^{m-1}) ={\bf 0} \implies {\bf x} = {\bf y}.$
\end{enumerate}
\end{definition}

\begin{definition}\rm \cite{MR3821072}
A tensor $\mathcal{A} \in \mathbb{T}(m,n)$ is said to be an $R_{0}$-tensor, if $\rm{SOL}(\mathcal{A},{\bf 0}) = \{{\bf 0}\}.$
\end{definition}

\begin{definition}\rm \cite[Definition 5.2]{Shao}\label{row-diagonal}
A tensor $\mathcal{A} = (a_{i_{1}i_{2}...i_{m}}) \in \mathbb{T}(m,n)$ is said to be a row-diagonal tensor, if for each $i \in [n]$, we have
$$a_{ii_{2}i_{3}...i_{m}} = {0}~(\text{if}~ i_{2}, i_{3},...,i_{m}~\text{are not all equal}).$$
\end{definition}

\begin{definition} \rm \label{PCP} \cite{Gowda}
Let ${\Theta} = (\mathcal{A}_{1}, \mathcal{A}_{2},...,\mathcal{A}_{m-1}) \in \Lambda(m,n)$ and ${\bf q} \in \mathbb{R}^{n}.$ The polynomial complementarity problem, denoted by $\rm{PCP}({\Theta},{\bf q})$, is to find a vector ${\bf x} \in \mathbb{R}^{n}$ satisfying
\begin{equation}\label{PCP1}
{\bf x} \geq {\bf 0},~{\sum_{k=1}^{m-1}}{\mathcal{A}_{k}}{\bf x}^{m-k} +{\bf q} \geq {\bf 0},~~\text{and}~~{\bf x}^{T}( {\sum_{k=1}^{m-1}}{\mathcal{A}_{k}}{\bf x}^{m-k}+ {\bf q}) = 0. 
\end{equation}
The set of all the vectors ${\bf x}$ satisfying Eq. (\ref{PCP1})  is said to be the solution set of $\rm{PCP}({\Theta},{\bf q})$, and is denoted by  $\mathrm{SOL}({\Theta},{\bf q}).$ We use the notation ${\mathrm{\Psi}({\bf x})}$ to denote ${\sum_{k=1}^{m-1}}{\mathcal{A}_{k}}{\bf x}^{m-k}$ for any ${\bf x} \in \mathbb{R}^{n}.$
\end{definition}

\begin{remark}\rm
It is to be noted that when $\mathcal{A}_{i}~(i = 2,3,...,m-1)$ in ${\Theta}$ are zero tensors, then the $\mathrm{PCP} ({\Theta},{\bf q})$ reduces to the $\mathrm{TCP}(\mathcal{A}_{1},{\bf q}).$  If  $\mathcal{A}_{i}~(i = 1,2,...,m-2)$ in ${\Theta}$ are zero tensors, then the $\mathrm{PCP} ({\Theta},{\bf q})$ reduces to the $\mathrm{LCP}(\mathcal{A}_{m-1},{\bf q}).$
\end{remark}

\begin{theorem}{\rm \cite[Proposition 2.1]{Gowda}}\label{compactness} Let ${\Theta} = (\mathcal{A}_{1}, \mathcal{A}_{2},...,\mathcal{A}_{m-1}) \in \Lambda(m,n)$, and ${\bf q} \in \mathbb{R}^{n}.$ If $\mathcal{A}_{1}$ is an ${R}_{0}$-tensor, then $\rm{PCP}({\Theta},{\bf q})$ has a compact solution set for any ${\bf q} \in \mathbb{R}^{n}.$ 
\end{theorem}
To examine the finiteness of the solution set of the PCP, we introduce the concepts of non-degenerate and strong non-degenerate tensor tuples, and explore their relationship with non-degenerate tensors in the following section (see Section \ref{Section 3}). Subsequently, we utilize these structured tensor tuples to analyse the finiteness of the  solution set of the PCP (refer to Section \ref{Section 4}). At last, we provide an equivalent condition for the finiteness of the solution set of the TCP in Section \ref{Section 5}.

\section{Non-degenerate and Strong Non-degenerate Tensor Tuples}\label{Section 3}
In this section, we first introduce the non-degenerate tensor tuple and strong non-degenerate tensor tuple, and then we investigate their relationships with the non-degenerate tensors. 
\begin{definition}\rm \label{non-degenerate tensor tuple}
A tensor tuple ${\Theta} = (\mathcal{A}_{1},\mathcal{A}_{2},...,\mathcal{A}_{m-1}) \in \Lambda(m,n)$ is said to be a non-degenerate tensor tuple if 
\begin{equation*}\label{ND tuple 1}
{\bf x} * {\Psi({\bf x})} = {\bf 0} \implies {\bf x} = {\bf 0}.
\end{equation*}
\end{definition}

\begin{definition}\rm \label{strong non-degenerate tensor tuple}
A tensor tuple ${\Theta} = (\mathcal{A}_{1},\mathcal{A}_{2},...,\mathcal{A}_{m-1}) \in \Lambda(m,n)$ is said to be a strong non-degenerate tensor tuple if 
\begin{equation*}\label{strong ND tuple 1}
({\bf x} - {\bf y}) * ({\Psi({\bf x})}-{\Psi({\bf y})} = {\bf 0} \implies {\bf x} = {\bf y}.
\end{equation*}  
\end{definition}

\begin{remark}\rm
From the above definitions, it is clear that a strong non-degenerate tensor tuple is  non-degenerate.
\end{remark}
  
\begin{remark}\rm\label{aa}
 Note that if $\mathcal{A}_{i}~(i = 2,...,m-1)$ in $\Theta$ are zero tensors, then $\Theta$ is a non-degenerate tensor tuple if and only if $\mathcal{A}_{1}$ is a non-degenerate tensor.
  If $\mathcal{A}_{i}~(i = 1,2,...,m-2)$ in $\Theta$ are zero tensors, then $\Theta$ is a non-degenerate tensor tuple if and only if $\mathcal{A}_{m-1}$ is a non-degenerate matrix. Moreover, if  $\mathcal{A}_{i}~(i = 2,...,m-1)$ in $\Theta$ are zero tensors, then $\Theta$ is a strong non-degenerate tensor tuple if and only if $\mathcal{A}_{1}$ is a strong non-degenerate tensor.
 \end{remark}
 
 \begin{remark}\rm\label{bb} The following observations are in order. 
 \begin{enumerate}
 \item[\rm(i)] Even if all the tensors $\mathcal{A}_{i}~(i =1,2,...,m-1)$ are non-degenerate tensors, ${\Theta} = (\mathcal{A}_{1}, \mathcal{A}_{2},...,\mathcal{A}_{m-1}) \in \Lambda(m,n)$ need not necessarily a non-degenerate tensor tuple (see Example \ref{All ND}).
 \item[\rm(ii)] ${\Theta} = (\mathcal{A}_{1}, \mathcal{A}_{2},...,\mathcal{A}_{m-1}) \in \Lambda(m,n)$ being a non-degenerate tensor tuple need not imply that all the tensors $\mathcal{A}_{i}~(i = 1,2,...,m-1)$ are non-degenerate tensors (refer to Example \ref{theta ND}). 
 \item[\rm(iii)] A  non-degenerate tensor tuple need not be a strong non-degenerate tensor tuple (see Example \ref{theta ND}).
 \end{enumerate}
 \end{remark}

 \begin{example}\rm \label{All ND}
 Let ${\Theta} = (\mathcal{A}_{1},\mathcal{A}_{2}) \in \Lambda(3,2)$, where $\mathcal{A}_{1} = (a^{(1)}_{ijk}) \in \mathbb{T}(3,2)$ such that $a^{(1)}_{111} =1, a^{(1)}_{122} = 1, a^{(1)}_{211} = 1, a^{(1)}_{222} = 1$ and other entries are zero, and $\mathcal{A}_{2} = (a^{(2)}_{ij}) \in \mathbb{T}(2,2)$ with $a^{(2)}_{11} = 1, a^{(2)}_{12} = 1, a^{(2)}_{21} = 0, a^{(2)}_{22} =4$. We show that \rm(i) $\mathcal{A}_{1}$ and $\mathcal{A}_{2}$ are non-degenerate, \rm(ii)  but ${\Theta}$ is not a non-degenerate tensor tuple. For any ${\bf x} = (x_{1},x_{2})^{T} \in \mathbb{R}^{2}$, we have $\mathcal{A}_{1}{\bf x}^{2} = (x_{1}^{2}+x_{2}^{2}, x_{1}^{2}+x_{2}^{2})^{T},~~\mathcal{A}_{2}{\bf x} = (x_{1}+x_{2},4x_{2})^{T},~\text{and}~{\Psi}({\bf x}) = (x_{1}^{2}+x_{2}^{2}+x_{1}+x_{2}, x_{1}^{2}+x_{2}^{2}+4x_{2})^{T}.$
\begin{enumerate}
\item[\rm(i)] Let ${\bf x} \in \mathbb{R}^{2}$ satisfy
${\bf x} * \mathcal{A}_{1}{\bf x}^{2} = {\bf 0}.$ This implies that
$\begin{bmatrix}
x_{1}(x_{1}^{2}+x_{2}^{2})\\
x_{2}(x_{1}^{2}+x_{2}^{2})
\end{bmatrix}  = \begin{bmatrix}
0 \\
0
\end{bmatrix}.$ This gives ${\bf x} = {\bf 0},$ and hence $\mathcal{A}_{1}$ is a non-degenerate tensor. Now for any ${\bf x} \in \mathbb{R}^{2},~{\bf x} * \mathcal{A}_{2}{\bf x} = {\bf 0}$  implies $\begin{bmatrix}
x_{1}(x_{1}+x_{2}) \\ 
4x_{2}^{2}
\end{bmatrix} = \begin{bmatrix}
0 \\
0
\end{bmatrix}.$ This implies that ${\bf x} = {\bf 0}$, and hence $\mathcal{A}_{2}$ is non-degenerate.
\item[\rm(ii)] Let ${\bf x} = (-1,0)^{T} \in \mathbb{R}^{2}.$ It is easy to see that the nonzero ${\bf x}$ satisfies ${\bf x} * {\Psi}({\bf x}) = {\bf 0},$ and hence ${\Theta}$ is not a non-degenerate tensor tuple.
\end{enumerate}
 \end{example}

\begin{example}\rm \label{theta ND}
 Let ${\Theta} = (\mathcal{A}_{1},\mathcal{A}_{2}) \in \Lambda(3,2)$, where $\mathcal{A}_{1} = (a^{(1)}_{ijk}) \in \mathbb{T}(3,2)$ such that $a^{(1)}_{111} =1, a^{(1)}_{122} = 1, a^{(1)}_{211} = 1, a^{(1)}_{222} = 1$ and other entries are zero, and $\mathcal{A}_{2} = (a^{(2)}_{ij}) \in \mathbb{T}(2,2)$ with $a^{(2)}_{11} = 0, a^{(2)}_{12} = 2, a^{(2)}_{21} = 0, a^{(2)}_{22} =0.$ We show that \rm(i) ${\Theta}$ is a non-degenerate tensor tuple, \rm(ii)  but $\mathcal{A}_{2}$ is not non-degenerate, and \rm(iii) ${\Theta}$ is not a strong non-degenerate tensor tuple. For any ${\bf x} = (x_{1},x_{2})^{T} \in \mathbb{R}^{2},$ we have $\mathcal{A}_{1}{\bf x}^{2} = (x_{1}^{2}+x_{2}^{2}, x_{1}^{2}+x_{2}^{2})^{T},~~\mathcal{A}_{2}{\bf x} = (2x_{2},0)^{T},~\text{and}~{\Psi}({\bf x}) = (x_{1}^{2}+x_{2}^{2}+2x_{2}, x_{1}^{2}+x_{2}^{2})^{T}.$ 
\begin{enumerate}
\item[\rm(i)] Let ${\bf x} \in \mathbb{R}^{2}$ satisfy ${\bf x} * {\Psi}({\bf x}) = {\bf 0}.$ This implies that
$$\begin{bmatrix}
x_{1}(x_{1}^{2}+x_{2}^{2}+2x_{2})\\
x_{2}(x_{1}^{2}+x_{2}^{2})
\end{bmatrix} = \begin{bmatrix}
0 \\
0
\end{bmatrix}.$$
By an easy verification, we get ${\bf x} = {\bf 0},$ and hence ${\Theta}$ is a non-degenerate tensor tuple.
\item[\rm(ii)] Note that ${\bf x} = (1,0)^{T} \in \mathbb{R}^{2}$ satisfies ${\bf x} * {\mathcal{A}_{2}}{\bf x} = {\bf 0}.$ Therefore $\mathcal{A}_{2}$ is not a non-degenerate tensor.
\item[\rm(iii)] Let ${\bf x} = (1,1)^{T}$ and ${\bf y} = (-1,1)^{T}$ be in $\mathbb{R}^{2}.$ Then we can see easily that $({\bf x} - {\bf y}) * ({\Psi}({\bf x})-{\Psi}({\bf y})) = {\bf 0}.$ Thus ${\Theta}$ is not a strong non-degenerate tensor tuple. 
\end{enumerate} 
\end{example}

From \cite[Proposition 4.1]{MR4310678}, it is evident that there does not exist an odd ordered strong non-degenerate tensor. However,  a strong non-degenerate tensor tuple need not always be of even order, as demonstrated below.

\begin{example}\rm\label{odd order}
Let ${\Theta} = (\mathcal{A}_{1},\mathcal{A}_{2}) \in \Lambda(3,2)$, where $\mathcal{A}_{1} = (a^{(1)}_{ijk}) \in \mathbb{T}(3,2)$ such that  $a^{(1)}_{211}  =1$ and other entries are zero, and $\mathcal{A}_{2} = (a^{(2)}_{ij}) \in \mathbb{T}(2,2)$ with $a^{(2)}_{11} = 1, a^{(2)}_{12} = 0, a^{(2)}_{21} = 0, a^{(2)}_{22} =1.$ Clearly, ${\Theta}$ is an odd-ordered tensor tuple. We show that ${\Theta}$ is  strong non-degenerate. For any ${\bf x} \in \mathbb{R}^{2}$, we have $\mathcal{A}_{1}{\bf x}^{2} = (0,x_{1}^{2})^{T}, \mathcal{A}_{2}{\bf x} = (x_{1},x_{2})^{T},$ and ${\Psi}({\bf x}) = (x_{1},x_{1}^{2}+x_{2})^{T}$. Let ${\bf x}, {\bf y} \in \mathbb{R}^{2}$ satisfy
\begin{eqnarray*}
\allowdisplaybreaks
& ({\bf x} - {\bf y})*({\Psi}({\bf x})-{\Psi}({\bf y})) = {\bf 0}, \notag \\
\implies & (x_{1}-y_{1})^{2} =0,~(x_{2}-y_{2})((x_{1}^{2}-y_{1}^{2})+(x_{2}-y_{2})) =0.
\end{eqnarray*}
This implies that ${\bf x} = {\bf y}$, and hence ${\Theta}$ is strong non-degenerate. 
\end{example}

\begin{remark}\rm\label{r11} We make the following observations.
\begin{enumerate}
\item[\rm(i)] ${\Theta} = (\mathcal{A}_{1}, \mathcal{A}_{2},...,\mathcal{A}_{m-1}) \in \Lambda(m,n)$ being a strong non-degenerate tensor tuple need not imply that all the tensors $\mathcal{A}_{i}~(i = 1,2,...,m-1)$ are strong non-degenerate tensors (refer to Example \ref{none are strong ND}).
 \item[\rm(ii)]  Even if all the even ordered tensors $\mathcal{A}_{i}~(i =1,2,...,m-1)$ in ${\Theta} = (\mathcal{A}_{1}, \mathcal{A}_{2},...,\\ \mathcal{A}_{m-1}) \in \Lambda(m,n)$ are strong non-degenerate tensors,  ${\Theta}$ is not necessarily a strong non-degenerate tensor tuple (see Example \ref{not strong tuple}). 
 \end{enumerate}
 \end{remark}

\begin{example}\rm\label{none are strong ND}
Let ${\Theta} = (\mathcal{A}_{1},\mathcal{A}_{2},\mathcal{A}_{3}) \in \Lambda(4,3),$ where $\mathcal{A}_{1} = (a^{(1)}_{ijkl}) \in \mathbb{T}(4,3)$ such that $a^{(1)}_{1111}=1, a^{(1)}_{3333} = 1$ and other entries are zero, $\mathcal{A}_{2} = (a^{(2)}_{ijk}) \in \mathbb{T}(3,3)$ with $a^{(2)}_{233} = 1$  and other entries being zero, and $\mathcal{A}_{3} = (a^{(3)}_{ij}) \in \mathbb{T}(2,3)$ such that $a^{(3)}_{22} = 2$ and rest of the entries being zero. We show that \rm(i) ${\Theta}$ is a strong non-degenerate tensor tuple, \rm(ii) but none of the tensors $\mathcal{A}_{1}, \mathcal{A}_{2}, \mathcal{A}_{3}$ are strong non-degenerate. For any ${\bf x} \in \mathbb{R}^{3},$ we have $\mathcal{A}_{1}{\bf x}^{3} = (x_{1}^{3},0,x_{3}^{3})^{T}, \mathcal{A}_{2}{\bf x}^{2} = (0,x_{3}^{2},0)^{T}, \mathcal{A}_{3}{\bf x} = (0,2x_{2},0)^{T},~\text{and}~{\Psi}({\bf x}) = (x_{1}^{3},x_{3}^{2}+2x_{2},x_{3}^{3})^{T}.$
\begin{enumerate}
\item[\rm(i)] Let ${\bf x},{\bf y} \in \mathbb{R}^{3}$ satisfy 
\begin{eqnarray*}
\allowdisplaybreaks
& ({\bf x}-{\bf y})*({\Psi}({\bf x})-{\Psi}({\bf y})) = {\bf 0}, \\
\implies &  \begin{bmatrix}
(x_{1}-y_{1})(x_{1}^{3}-y_{1}^{3})\\
(x_{2}-y_{2})[(x_{3}^{2}+2x_{2})-(y_{3}^{2}+2y_{2})]\\
(x_{3}-y_{3})(x_{3}^{3}-y_{3}^{3})
\end{bmatrix} = \begin{bmatrix}
0 \\0 \\ 0
\end{bmatrix}.
\end{eqnarray*}
A simple calculation yields ${\bf x} = {\bf y}$. Thus ${\Theta}$ is a strong non-degenerate tensor tuple.
\item[\rm(ii)] Note that ${\bf x} = (1,1,1)^{T}$ and ${\bf y} = (1,0,1)^{T}$ satisfy $({\bf x}-{\bf y})* (\mathcal{A}_{1}{\bf x}^{3} - \mathcal{A}_{1}{\bf y}^{3}) = {\bf 0}.$ Hence $\mathcal{A}_{1}$ cannot be a strong non-degenerate tensor. Since a strong non-degenerate tensor is always of even order, $\mathcal{A}_{2}$ is not a strong non-degenerate tensor. Also, by taking ${\bf x} = (1,1,0)^{T}$ and ${\bf y} = (0,1,1)^{T}$, we can see that $({\bf x}-{\bf y})* (\mathcal{A}_{3}{\bf x}-\mathcal{A}_{3}{\bf y}) = {\bf 0},$ and hence $\mathcal{A}_{3}$ is not a strong non-degenerate tensor.   
\end{enumerate}   
\end{example}

\begin{example}\rm\label{not strong tuple}
Let ${\Theta} = (\mathcal{A}_{1},\mathcal{A}_{2},\mathcal{A}_{3}) \in \Lambda(4,2),$ where $\mathcal{A}_{1} = (a^{(1)}_{ijkl}) \in \mathbb{T}(4,2)$ such that $a^{(1)}_{1111} =1, a^{(1)}_{2222}=-1, a^{(1)}_{1222} =1$ and other entries are zero, $\mathcal{A}_{2}  = (a^{(2)}_{ijk}) \in \mathbb{T}(3,2)$ with $a^{(2)}_{111} =1$ and other entries being zero, and $\mathcal{A}_{3} = (a^{(3)}_{ij}) \in \mathbb{T}(2,2)$ such that $a^{(3)}_{11} =1, a^{(3)}_{22}=1$ and other entries are zero. We show that \rm(i) $\mathcal{A}_{1}$ and $\mathcal{A}_{3}$ are strong non-degenerate tensors, \rm(ii) but  ${\Theta}$ is not a strong non-degenerate tensor tuple. For any ${\bf x} \in \mathbb{R}^{2}$, we have $\mathcal{A}_{1}{\bf x}^{3} = (x_{1}^{3}+x_{2}^{3},-x_{2}^{3})^{T}, \mathcal{A}_{2}{\bf x}^{2} = (x_{1}^{2},0)^{T}, \mathcal{A}_{3}{\bf x} = (x_{1},x_{2})^{T},~\text{and}~{\Psi}({\bf x}) = (x_{1}^{3}+x_{2}^{3}+x_{1}^{2}+x_{1},-x_{2}^{3}+x_{2})^{T}.$
\begin{enumerate}
\item[\rm(i)] Let ${\bf x},{\bf y}$ in $\mathbb{R}^{2}$ satisfy $({\bf x}-{\bf y})*(\mathcal{A}_{1}{\bf x}^{3}-\mathcal{A}_{1}{\bf y}^{3}) = {\bf 0}.$ This implies that
$$(x_{1}-y_{1})[(x_{1}^{3}+x_{2}^{3})-(y_{1}^{3}+y_{2}^{3})] =0,~\text{and}~(x_{2}-y_{2})(y_{2}^{3}-x_{2}^{3})=0.$$
By an easy calculation, we get ${\bf x} = {\bf y}$, and hence $\mathcal{A}_{1}$ is a strong non-degenerate tensor. We now show that $\mathcal{A}_{3}$ is strong non-degenerate. Let ${\bf x},{\bf y}$ in $\mathbb{R}^{2}$ satisfy $({\bf x}-{\bf y})*(\mathcal{A}_{3}{\bf x}-\mathcal{A}_{3}{\bf y}) = {\bf 0}.$ This gives $(x_{1}-y_{1})^{2}=0,~\text{and}~(x_{2}-y_{2})^{2}=0$. This implies that ${\bf x} = {\bf y},$ and therefore $\mathcal{A}_{3}$ is  strong non-degenerate.
\item[\rm(ii)] Note that ${\bf x} = (0,1)^{T}$ and ${\bf y} =(0,-1)^{T}$ satisfy $({\bf x}-{\bf y})*({\Psi}({\bf x})-{\Psi}({\bf y}))={\bf 0}.$ Thus ${\Theta}$ is not a strong non-degenerate tensor tuple.
\end{enumerate} 
\end{example}
    
In the following, we show that the principal subtensor tuple of a non-degenerate tensor tuple is non-degenerate. Before proceeding further, we recall the definition of a principal subtensor tuple from \cite{li2024strict}.
\begin{definition}\rm\cite[Definition 4]{li2024strict}
Let ${\Theta} = (\mathcal{A}_{1},\mathcal{A}_{2},...,\mathcal{A}_{m-1}) \in \Lambda(m,n),$ and $I$ be a subset of $[n].$ Then ${\Theta}(I) = (\mathcal{A}_{1}(I), \mathcal{A}_{2}(I),...,\mathcal{A}_{m-1}(I)) \in {\Lambda}(m,\lvert I \rvert)$ is called the principal subtensor tuple of ${\Theta},$ where $\mathcal{A}_{j}(I)$ is a principal subtensor of $\mathcal{A}_{j},$ for each $j \in [m-1].$ 
\end{definition}

\begin{proposition}
Any principal subtensor tuple of a non-degenerate tensor tuple  is non-degenerate.  
\end{proposition}
\begin{proof}
Let ${\Theta} = (\mathcal{A}_{1},\mathcal{A}_{2},...,\mathcal{A}_{m-1}) \in \Lambda(m,n)$ be a non-degenerate tensor tuple and $I$ be a subset of $[n].$ Then the corresponding principal subtensor tuple of ${\Theta}$ is ${\Theta}(I) = (\mathcal{A}_{1}(I),\mathcal{A}_{2}(I),...,\mathcal{A}_{m-1}(I)) \in {\Lambda}(m, \lvert I \rvert).$  Suppose that there exists ${\bf x} \in \mathbb{R}^{\lvert I \rvert}$ such that ${\bf x} * {\Psi}_{I}({\bf x}) ={\bf 0},$ where ${\Psi}_{I}({\bf x}) = {\displaystyle\sum_{k=1}^{m-1}}{\mathcal{A}_{k}}(I){\bf x}^{m-k}.$ Define ${\bf y} \in \mathbb{R}^{n}$ such that $y_{i} = x_{i},$ for all $i \in I$, and zero otherwise. Then, we have
\begin{equation*}\label{pricipal}
\begin{aligned}
{\bf y} * {\Psi}({\bf y}) &= {\bf y}* (\mathcal{A}_{1}{\bf y}^{m-1}+\mathcal{A}_{2}{\bf y}^{m-2}+...+\mathcal{A}_{m-1}{\bf y}) \\
& ={\bf x} * (\mathcal{A}_{1}(I){\bf x}^{m-1}+\mathcal{A}_{2}(I){\bf x}^{m-2}+...+\mathcal{A}_{m-1}(I){\bf x}) \\
&= {\bf x} * {\Psi}_{I}({\bf x})\\
&= {\bf 0}. 
\end{aligned}
\end{equation*}
This implies that ${\bf y} * {\Psi}({\bf y}) = {\bf 0}.$ Since ${\Theta}$ is a non-degenerate tensor tuple, we get ${\bf y} = {\bf 0}$, and hence ${\bf x} = {\bf 0}$. Therefore ${\Theta}(I)$ is non-degenerate. 
\end{proof}

\section{Finiteness Property for the PCP}\label{Section 4}

In this section, we investigate the finiteness property of the non-degenerate tensor tuple ${\Theta}$, and  establish a sufficient condition for the finiteness property of ${\Theta}$, when ${\Theta}$ is strong non-degenerate. Before proceeding further, we formally define the finiteness property for a tensor tuple in the following.

\begin{definition}\rm\label{Finiteness}
A tensor tuple ${\Theta} = (\mathcal{A}_{1},\mathcal{A}_{2},...,\mathcal{A}_{m-1}) \in \Lambda(m,n)$ is said to have the finiteness property if $\rm{SOL}({\Theta},{\bf q})$ is a finite set for all ${\bf q} \in \mathbb{R}^{n}$.
\end{definition}

\begin{theorem}\label{tensors}{\rm \cite[Theorem 3.1]{MR4310678}}
Let $\mathcal{A} \in \mathbb{T}(m,n)$ be an even ordered row-diagonal tensor. Then $\mathcal{A}$ is non-degenerate if and only if $\rm{TCP}({\mathcal{A}},{\bf q})$ has a finite solution set for all ${\bf q} \in \mathbb{R}^{n}.$
\end{theorem} 
 
In the following, we investigate whether an analogue of Theorem \ref{tensors} holds for a tensor tuple $\Theta \in \Lambda(m, n)$; that is, whether the equivalence between a non-degenerate tensor tuple and the finiteness property remains valid when all tensors $\mathcal{A}_1, \mathcal{A}_2,...,$ $ \mathcal{A}_{m-1}$ in ${\Theta}$ are row-diagonal and $m$ is even. Note that $\mathcal{A}_{m-1}$ in $\Theta$ is a matrix of order $n$, and therefore it is trivially a row-diagonal tensor. The following examples illustrate that this result need to  be true.

\begin{example}\rm\label{row-diagonal ND but not finite}
Let ${\Theta} = (\mathcal{A}_{1},\mathcal{A}_{2},\mathcal{A}_{3}) \in \Lambda(4,3),$ where $\mathcal{A}_{1} = (a^{(1)}_{ijkl}) \in \mathbb{T}(4,3)$ such that $a^{(1)}_{1111} =1$ and other entries are zero, $\mathcal{A}_{2}  = (a^{(2)}_{ijk}) \in \mathbb{T}(3,3)$ with $a^{(2)}_{222} =1, a^{(2)}_{233} = 1, a^{(2)}_{333} =1, a^{(2)}_{322} =1$ and other entries being zero, and $\mathcal{A}_{3} = (a^{(3)}_{ij}) \in \mathbb{T}(2,3)$ such that $a^{(3)}_{11} =1$ and other entries are zero. It is easy to see that $\mathcal{A}_{1}, \mathcal{A}_{2}$  and $\mathcal{A}_{3}$ are row-diagonal tensors. We show that \rm(i) ${\Theta}$ is a non-degenerate tensor tuple, \rm(ii) but $\rm{SOL}({\Theta},{\bf q})$ is  not a finite set. For any ${\bf x} \in \mathbb{R}^{3}$, we have $\mathcal{A}_{1}{\bf x}^{3} = (x_{1}^{3},0,0)^{T}, \mathcal{A}_{2}{\bf x}^{2} = (0,x_{2}^{2}+x_{3}^{2}, x_{2}^{2}+x_{3}^{2})^{T}, \mathcal{A}_{3}{\bf x} = (x_{1},0,0)^{T},~\text{and}~\Psi({\bf x}) = (x_{1}^{3}+x_{1},x_{2}^{2}+x_{3}^{2},x_{2}^{2}+x_{3}^{2})^{T}.$
\begin{enumerate}
\item[\rm(i)] Let ${\bf x} \in \mathbb{R}^{3}$ satisfy ${\bf x} * {\Psi}({\bf x}) = {\bf 0}.$ This implies that
$$x_{1}(x_{1}^{3}+x_{1}) = 0,~ x_{2}(x_{2}^{2}+x_{3}^{2})=0,~\text{and}~x_{3}(x_{2}^{2}+x_{3}^{2}) = 0.$$
By as easy calculation, we get ${\bf x} = {\bf 0}$, and hence $\Theta$ is a non-degenerate tensor tuple.
\item[\rm(ii)] Take ${\bf q} = (0,-1,-1)^{T}$ in $\mathbb{R}^{3}$. Observe that the vector ${\bf x} = (0,\cos{t},\sin{t})^{T} \in \mathbb{R}^{3}$ satisfies
$${\bf x} \geq {\bf 0},~{\Psi}({\bf x})+{\bf q} \geq {\bf 0}~\text{and}~{\bf x}^{T}({\Psi}({\bf x})+{\bf q}) =0,$$
for any $0 \leq t \leq \frac{\pi}{2}.$ Therefore $\rm{SOL}({\Theta},{\bf q})$ is not a finite set, and hence ${\Theta}$ does not have the finiteness property. 
\end{enumerate} 
\end{example}

\begin{example}\rm\label{row-diagonal finite but not ND}
Let ${\Theta} = (\mathcal{A}_{1},\mathcal{A}_{2},\mathcal{A}_{3}) \in \Lambda(4,2),$ where $\mathcal{A}_{1} = (a^{(1)}_{ijkl}) \in \mathbb{T}(4,2)$ such that $a^{(1)}_{1111} =1, a^{(1)}_{2222}=-1, a^{(1)}_{1222} =1$ and other entries are zero, $\mathcal{A}_{2}  = (a^{(2)}_{ijk}) \in \mathbb{T}(3,2)$ with $a^{(2)}_{111} =1$ and other entries being zero, and $\mathcal{A}_{3} = (a^{(3)}_{ij}) \in \mathbb{T}(2,2)$ such that $a^{(3)}_{11} =1, a^{(3)}_{22}=1$ and other entries are zero. It is easy to see that $\mathcal{A}_{1}, \mathcal{A}_{2}$  and $\mathcal{A}_{3}$ are row-diagonal tensors. We show that \rm(i) ${\Theta}$ is not a non-degenerate tensor tuple, \rm(ii) but $\Theta$ has the finiteness property. For any ${\bf x} \in \mathbb{R}^{2}$, we have $\mathcal{A}_{1}{\bf x}^{3} = (x_{1}^{3}+x_{2}^{3},-x_{2}^{3})^{T}, \mathcal{A}_{2}{\bf x}^{2} = (x_{1}^{2},0)^{T}, \mathcal{A}_{3}{\bf x} = (x_{1},x_{2})^{T},~\text{and}~{\Psi}({\bf x}) = (x_{1}^{3}+x_{2}^{3}+x_{1}^{2}+x_{1},-x_{2}^{3}+x_{2})^{T}.$ 
\begin{enumerate}
\item[\rm(i)] Note that the nonzero vector ${\bf x} = (0,1)^{T} \in \mathbb{R}^{2}$ satisfies ${\bf x} * {\Psi}({\bf x}) = {\bf 0}.$ Thus ${\Theta}$ is not non-degenerate.
\item[\rm(ii)] We now show that ${\Theta}$ has the finiteness property. If $\rm{SOL}({\Theta},{\bf q})$ is empty, then we are done. Suppose that $\rm{SOL}({\Theta},{\bf q}) \neq \emptyset$, and ${\bf x} = (x_{1},x_{2})^{T}$ be in $\rm{SOL}({\Theta},{\bf q}).$ This implies that
\begin{equation}\label{row-diag1}
{\bf x} \geq {\bf 0},~{\bf y}:={\Psi}({\bf x}) +{\bf q} \geq {\bf 0},~\text{and}~{\bf x}^{T}{\bf y} =0.
\end{equation}
We consider the following cases:

\textbf{{Case 1.}} If $x_{1} > 0$ and $x_{2} = 0$, then from Eq. (\ref{row-diag1}), we get $y_{1} = 0$. This gives $x_{1}^{3}+x_{1}^{2}+x_{1}+q_{1} = 0.$ It can be easily verified that  $x_{1}^{3}+x_{1}^{2}+x_{1}+q_{1}$ is a nonzero polynomial for any ${\bf q} \in \mathbb{R}^{2}.$ Since any nonzero polynomial has finitely many zeroes, we get finitely many values of $x_{1}.$

\textbf{{Case 2.}} If $x_{1}=0$ and $x_{2}>0$, then from Eq. (\ref{row-diag1}), we get $y_{2} =0$. This implies that $-x_{2}^{3}+x_{2}+q_{2} = 0,$ and hence we get finitely many values of $x_{2}$ for any ${\bf q} \in \mathbb{R}^{2}.$

\textbf{{Case 3.}} If $x_{1} > 0$ and $x_{2} > 0$, then from Eq. (\ref{row-diag1}), we get $y_{1} = 0$ and $y_{2}=  0$. This implies that 
$$x_{1}^{3}+x_{2}^{3}+x_{1}^{2}+x_{1} + q_{1}=0,~\text{and}~-x_{2}^{3}+x_{2}+q_{2} = 0.$$
From $x_{1}^{3}+x_{2}^{3}+x_{1}^{2}+x_{1} + q_{1}=0$, we get $x_{2} = -(x_{1}^{3}+x_{1}^{2}+x_{1}+q_{1})^{\frac{1}{3}}.$ Substituting this value of $x_{2}$ in $-x_{2}^{3}+x_{2}+q_{2} = 0$, and upon simplifying, we get $p(x_{1}) = 0,$ where  
$$p(x_{1})=(x_{1}^{3}+x_{1}^{2}+x_{1} + q_{1})-(x_{1}^{3}+x_{1}^{2}+x_{1} + (q_{1}+q_{2}))^{3}.$$
Since $p(x_{1})$ is a nonzero polynomial for any value of ${\bf q} \in \mathbb{R}^{2}$, it has finitely many zeroes. Thus we get finitely many values of $x_{1}$, and hence $x_{2}$.\\
From all of the above cases, we get $\rm{SOL}({\Theta},{\bf q})$ is a finite set for any ${\bf q} \in \mathbb{R}^{2}$, and thus ${\Theta}$ has the finiteness property. 
\end{enumerate}
\end{example}

A further question arises that when the involved tensor tuple is strong non-degenerate,  can we get the  finiteness property of the tensor tuple ${\Theta} \in \Lambda(m,n)$.  In the following theorem, we provide a sufficient condition for the finiteness property of the tensor tuple ${\Theta} \in \Lambda(m,n)$. 
\begin{theorem}\label{finiteness}
Let ${\Theta} = (\mathcal{A}_{1}, \mathcal{A}_{2},...,\mathcal{A}_{m-1}) \in \Lambda(m,n),$ ${\bf q} \in \mathbb{R}^{n},$ and $\mathcal{A}_{1}$ be an $R_{0}$-tensor. If ${\Theta}$ is a strong non-degenerate tensor tuple, then ${\Theta}$ has the finiteness property.
\end{theorem}
\begin{proof}
Given that ${\Theta} = (\mathcal{A}_{1}, \mathcal{A}_{2},...,\mathcal{A}_{m-1}) \in \Lambda(m,n),$ and  ${\bf q} \in \mathbb{R}^{n}.$ If $\rm{SOL}({\Theta},{\bf q})$ is empty, then there is nothing to prove. Assume that  $\rm{SOL}({\Theta},{\bf q}) \neq \emptyset$ and an infinite set for some ${\bf q} \in \mathbb{R}^{n}.$ Let $\{{\bf x}^{(k)}\}_{k =1}^{\infty}$ be a sequence in  $\rm{SOL}({\Theta},{\bf q})$. Since $\mathcal{A}_{1}$ is an ${R}_{0}$-tensor, therefore from Theorem \ref{compactness}, it follows that $\rm{SOL}({\Theta},{\bf q})$ is a compact set. So there exists a convergent subsequence of  $\{{\bf x}^{(k)}\}$ which converges to some ${\bf x} \in \rm{SOL}({\Theta},{\bf q}).$  Assume (without loss of generality) that the convergent subsequence of $\{{\bf x}^{(k)}\}$ is itself. We claim that
$$({\bf x}^{(k)}-{\bf x})* ({\Psi}({\bf x}^{(k)})-{\Psi}({\bf x})) = 0,~~\text{for all}~k \geq {k_{0}}.$$
We consider the following cases:\\
\textbf{{Case 1.}} When $x_{i} = 0$ and $({\Psi}({\bf x})+{\bf q})_{i} = 0.$ Then for all $k \geq 1$, we obtain 
$$({\bf x}^{(k)}-{\bf x})_{i}({\Psi}({\bf x}^{(k)})-{\Psi}({\bf x}))_{i}=({\bf x}^{(k)}-{\bf x})_{i} [({\Psi}({\bf x}^{(k)}) + {\bf q})-({\Psi}({\bf x})+{\bf q})]_{i} = 0.$$
\textbf{{Case 2.}} When $x_{i} = 0$ and $({\Psi}({\bf x})+{\bf q})_{i} > 0.$ Since $\{{\bf x}^{(k)}\}$ converges to ${\bf x}$ and ${\Psi}({\bf x})+{\bf q}$ is a continuous function, $\{{\Psi}({\bf x}^{(k)})+{\bf q}\}$  converges to $({\Psi}({\bf x})+{\bf q}).$ As $({\Psi}({\bf x})+{\bf q})_{i} > 0$, there exists  some $k_{1}$ such that $({\Psi}({\bf x}^{(k)})+{\bf q})_{i} > {0}$ for all $k \geq {k_{1}}.$ This implies that $x^{(k)}_{i} = 0$ for all $k \geq {k_{1}}.$ Hence, for all $k \geq {k_{1}},$ we have
$$({\bf x}^{(k)}-{\bf x})_{i}({\Psi}({\bf x}^{(k)})-{\Psi}({\bf x}))_{i}=({\bf x}^{(k)}-{\bf x})_{i} [({\Psi}({\bf x}^{(k)}) + {\bf q})-({\Psi}({\bf x})+{\bf q})]_{i} = 0.$$
\textbf{{Case 3.}} When $x_{i} > 0$ and $({\Psi}({\bf x})+{\bf q})_{i} = 0.$ Since $\{{\bf x}^{(k)}\}$ converges to ${\bf x}$ and $x_{i} > 0$, so there exists some $k_{2}$ such that $x^{(k)}_{i} > 0,$ for all $k \geq {k_{2}}.$ This implies that $({\Psi}({\bf x}^{(k)})+{\bf q})_{i} = 0,$  for all $k \geq {k_{2}}.$ Hence, for all $k \geq {k_{2}},$ we have
$$({\bf x}^{(k)}-{\bf x})_{i}({\Psi}({\bf x}^{(k)})-{\Psi}({\bf x}))_{i}=({\bf x}^{(k)}-{\bf x})_{i} [({\Psi}({\bf x}^{(k)}) + {\bf q})-({\Psi}({\bf x})+{\bf q})]_{i} = 0.$$
Let ${k_{0}} = \max\{1,{k_{1}},{k_{2}}\}$. Then from all of the above cases, we get 
$$({\bf x}^{(k)}-{\bf x})_{i}({\Psi}({\bf x}^{(k)})-{\Psi}({\bf x}))_{i} = 0,$$
for any $i \in [n]$ and for all $k \geq {k_{0}}.$ Since ${\Theta}$ is a strong non-degenerate tensor tuple, we get $\{{\bf x}^{(k)}\} = {\bf x},$ for all $k \geq {k_{0}}.$ Thus the sequence $\{{\bf x}^{(k)}\}$ is an eventually constant sequence. As  $\rm{SOL}({\Theta},{\bf q})$ is a compact set, we obtain  $\rm{SOL}({\Theta},{\bf q})$ is a finite set for all ${\bf q} \in \mathbb{R}^{n}.$ Hence ${\Theta} $ has the finiteness property. 
\end{proof}

The following example demonstrates that the converse of the Theorem \ref{finiteness} does not hold.
\begin{example}\rm\label{not ND but finite}
Let ${\Theta} = (\mathcal{A}_{1},\mathcal{A}_{2},\mathcal{A}_{3}) \in \Lambda(4,2),$ where $\mathcal{A}_{1} = (a^{(1)}_{ijkl}) \in \mathbb{T}(4,2)$ such that $a^{(1)}_{1111} =1, a^{(1)}_{1222} = 1, a^{(1)}_{2221} = -1, a^{(1)}_{2222} = -1$ and other entries are zero, $\mathcal{A}_{2}  = (a^{(2)}_{ijk}) \in \mathbb{T}(3,2)$ with $a^{(2)}_{111} =1$ and other entries being zero, and $\mathcal{A}_{3} = (a^{(3)}_{ij}) \in \mathbb{T}(2,2)$ such that $a^{(3)}_{11} =1$ and other entries are zero. We show that \rm(i) $\mathcal{A}_{1}$ is an $R_{0}$-tensor, \rm(ii) $\rm{SOL}({\Theta},{\bf q})$ is  a finite set for all ${\bf q} \in \mathbb{R}^{2}$ , \rm(iii) but ${\Theta}$ is not a strong non-degenerate tensor tuple. For any ${\bf x} \in \mathbb{R}^{2}$, we have 
$\mathcal{A}_{1}{\bf x}^{3} = (x_{1}^{3}+x_{2}^{3}, -x_{1}x_{2}^{2}-x_{2}^{3})^{T},~\mathcal{A}_{2}{\bf x}^{2} = (x_{1}^{2},0)^{T},~\mathcal{A}_{3}{\bf x} = (x_{1},0)^{T},~\text{and}~{\Psi}({\bf x}) = (x_{1}^{3}+x_{2}^{3}+x_{1}^{2}+x_{1},-x_{1}x_{2}^{2}-x_{2}^{3})^{T}.$
\begin{enumerate}
\item[\rm(i)] Let ${\bf x} \in \rm{SOL}(\mathcal{A}_{1},{\bf 0}).$ This implies that ${\bf x} \geq {\bf 0}, \mathcal{A}_{1}{\bf x}^{3} \geq {\bf 0}$ and ${\bf x}*\mathcal{A}_{1}{\bf x}^{3} = 0.$ This gives
\begin{equation*}
\begin{aligned}
& x_{1} \geq {0}, x_{2} \geq {0}, x_{1}^{3}+x_{2}^{3} \geq {0},
 -x_{1}x_{2}^{2}-x_{2}^{3} \geq {0},\\
 & ~\text{and}~x_{1}(x_{1}^{3}+x_{2}^{3}) = 0, x_{2}(-x_{1}x_{2}^{2}-x_{2}^{3})=0.
\end{aligned}  
\end{equation*}
By an easy calculation, we get ${\bf x} = {\bf 0}$. Thus  $\mathcal{A}_{1}$ is an $R_{0}$-tensor.
\item[\rm(ii)] We now show that ${\Theta}$ has the finiteness property. If $\rm{SOL}({\Theta},{\bf q})$ is an empty set, then we are done. Suppose that $\rm{SOL}({\Theta},{\bf q}) \neq \emptyset$, and ${\bf x} = (x_{1},x_{2})^{T}$ be in $\rm{SOL}({\Theta},{\bf q}).$ This implies that
\begin{eqnarray}\label{ff1}
\allowdisplaybreaks
& {\bf x} \geq {\bf 0},~{\bf y}:={\Psi}({\bf x})+{\bf q} \geq {\bf 0}~\text{and}~{\bf x}^{T}{\bf y} =0,  \notag \\
&~\text{where}~{\bf x} = \begin{bmatrix}
x_{1}\\x_{2}
\end{bmatrix},~\text{and}~ {\bf y} = \begin{bmatrix}
x_{1}^{3}+x_{2}^{3}+x_{1}^{2}+x_{1}+q_{1} \\
 -x_{1}x_{2}^{2}-x_{2}^{3}+q_{2}
\end{bmatrix}. 
\end{eqnarray}
We consider the following cases:

\textbf{{Case 1.}} If $x_{1} >0$ and $x_{2} =0$, then from Eq. (\ref{ff1}), we obtain  
\begin{equation}\label{nonzero polynomial}
x_{1}^{3} + x_{1}^{2}+x_{1}+q_{1} =0.
\end{equation}
It can be easily verified that $x_{1}^{3} + x_{1}^{2}+x_{1}+q_{1}$ is a nonzero polynomial for any value of ${\bf q} \in \mathbb{R}^{2}.$ Since any nonzero polynomial has finitely many zeros, upon solving Eq. (\ref{nonzero polynomial}), we get finitely many values of $x_{1}$.

\textbf{{Case 2.}} If $x_{2} >0$ and $x_{1} =0$, then from Eq. (\ref{ff1}), we get  $-x_{2}^{3} + q_{2} = 0$. This implies that ${\bf x} = (0,(q_{2})^{\frac{1}{3}})^{T}.$ 

\textbf{{Case 3.}} If $x_{1} > {0}$ and $x_{2} > 0$, then 
$$x_{1}^{3}+x_{2}^{3}+x_{1}^{2}+x_{1}+q_{1} = 0,~ -x_{1}x_{2}^{2}-x_{2}^{3}+q_{2} = 0.$$
Upon simplifying the above equation, we obtain a polynomial $x_{1}^{3}(x_{1}^{3}+x_{1}^{2}+x_{1}+q_{1})^{2} - (x_{1}^{3}+x_{1}^{2}+x_{1}+q_{1}+q_{2})^{3},$ in terms of $x_{1}.$  It can be easily verified that this is a nonzero polynomial for any value of ${\bf q} \in \mathbb{R}^{2}$. Since any nonzero polynomial in one variable has finitely many zeroes, we get finitely many values of $x_{1}$ and hence $x_{2}$.\\
From all of the above cases, we get $\rm{SOL}({\Theta},{\bf q})$ is a finite set for any ${\bf q} \in \mathbb{R}^{2}$, and thus ${\Theta}$ has the finiteness property.
\item[\rm(iii)] Let ${\bf x} = (-1,1)^{T}$ and ${\bf y} = (0,0)^{T}$. Then we see that ${\Psi} ({\bf x}) = {\bf 0}$ and ${\Psi} ({\bf y}) = {\bf 0}$. Thus we have $({\bf x}-{\bf y})* ({\Psi}({\bf x})-{\Psi}({\bf y})) = {\bf 0},$ but ${\bf x} \neq {\bf y}$. Hence  ${\Theta}$ is not strong non-degenerate.
\end{enumerate}
\end{example}

\section{A Result Related to the Finiteness of SOL$(\mathcal{A},{\bf q})$}\label{Section 5}
Palpandi et al. \cite{MR4310678} established that, for the class of even-order row-diagonal tensors, the set \rm{SOL}$(\mathcal{A},{\bf q})$ is finite if and only if $\mathcal{A}$ is non-degenerate. In this section, we present a different class of tensors for which the finiteness of the solution set of the \rm{TCP}$(\mathcal{A},{\bf q})$ is  equivalent to the  tensor $\mathcal{A}$ being non-degenerate. Prior to stating our main result, we introduce the following definition.

\begin{definition}\rm\label{matrix based tensor}\cite[Page 234]{Gowda} A tensor $\mathcal{A} \in \mathbb{T}(m,n)$ with $m$ being even, is said to be a matrix based tensor if there exists a real square matrix $\hat{\bf A}$ of order $n$ such that
$$\mathcal{A}{\bf x}^{m-1} = ({\hat{\bf A}}{\bf x})^{[m-1]}.$$
We say that $\mathcal{A}$ is a matrix based tensor corresponding to the matrix $\hat{\bf A}$ and an odd exponent $k~(= m-1).$
\end{definition}

\begin{lemma}\label{solution set} \rm{\cite[Theorem 4.1]{Gowda}} For a matrix based tensor $\mathcal{A} \in \mathbb{T}(m,n)$ corresponding to a matrix ${\hat{\bf A}} \in \mathbb{R}^{n \times n}$ and an odd exponent $k~(= m-1)$, we have 
$$\rm{SOL}(\mathcal{A},{\bf q}) = \rm{SOL}({\hat{\bf A}},{{\bf q}}^{[{\frac{1}{k}]}}).$$ 
\end{lemma}

\begin{lemma}\label{for LCP} \rm{\cite[Theorem 3.6.3]{MR3396730}}
For a given matrix  ${\bf M} \in \mathbb{R}^{n \times n},$ the following are equivalent.
\begin{enumerate}
\item[\rm(i)] ${\bf M}$ is a non-degenerate matrix.
\item[\rm(ii)] For all ${\bf q} \in \mathbb{R}^{n}$, \rm{SOL}$({\bf M},{\bf q})$ is a finite set.
\end{enumerate} 
\end{lemma}

We now provide our main result.
\begin{theorem}\label{matrix based 1}
Let $m$ be even and $\mathcal{A} \in \mathbb{T}(m,n)$ be a matrix based tensor corresponding to a matrix $\hat{\bf A}$ and an odd exponent $k~(= m-1).$ Then the followings are equivalent.
\begin{enumerate}
\item[\rm(i)] $\mathcal{A}$ is non-degenerate.
\item[\rm(ii)] $\hat{\bf A}$ is non-degenerate.
\item[\rm(iii)] \rm{SOL}$(\hat{\bf {A}},{\bf q})$ is a finite set for all ${\bf q} \in \mathbb{R}^{n}.$
\item[\rm(iv)]  \rm{SOL}$(\mathcal {A},{\bf q})$ is a finite set for all ${\bf q} \in \mathbb{R}^{n}.$
\end{enumerate}
\end{theorem}
\begin{proof} Given that $\mathcal{A} \in \mathbb{T}(m,n)$ is a matrix based tensor corresponding to $\hat{\bf A}$ and an odd exponent $k~(= m-1).$ This implies that 
\begin{equation}\label{equation}
\mathcal{A}{\bf x}^{m-1} = (\hat{\bf A} {\bf x})^{[m-1]}.
\end{equation}
(i) $\iff$ (ii): Using Eq. (\ref{equation}), for any ${\bf x} \in \mathbb{R}^{n},$ we obtain
\begin{eqnarray}\label{EQ2}
\allowdisplaybreaks
& {\bf x} * {\hat{\bf A}}{\bf x} = {\bf 0} 
\iff & {\bf x} * ({\hat{\bf A}}{\bf x})^{[m-1]} = {\bf 0} \iff {\bf x} * \mathcal{A}{\bf x}^{m-1} = {\bf 0}.
\end{eqnarray}
From Eq. (\ref{EQ2}), it follows that $\mathcal{A}$ is non-degenerate if and only if $\hat{\bf A}$ is non-degenerate.\\
(ii) $\iff$ (iii): This follows from Lemma \ref{for LCP}.\\
(iii) $\iff$ (iv): From Lemma \ref{solution set}, we have \rm{SOL}$(\mathcal{A},{\bf q})$ = \rm{SOL}$(\hat{\bf A},{\bf q}^{[\frac{1}{m-1}]}).$ As $m$ is even, our result is immediate.
\end{proof}

The following examples demonstrate that the class of matrix-based tensors is distinct from the class of row-diagonal tensors. Note that a matrix-based tensor must be of even order, whereas a row-diagonal tensor can be of arbitrary order.

\begin{example}\rm \label{row-diagonalbut not matrix based}
Let $\mathcal{A} = (a_{ijkl}) \in \mathbb{T}(4,2)$ such that $a_{1111} =1, a_{2222} =-1, a_{1222} =1$ and other entries are zero. It is easy to observe that $\mathcal{A}$ is a row-diagonal tensor. For any ${\bf x} \in \mathbb{R}^{2}$, we have $\mathcal{A}{\bf x}^{3} = (x_{1}^{3}+x_{2}^{3}, -x_{2}^{3})^{T}.$ In order to $\mathcal{A}$ to be a matrix based tensor, there must exist a matrix $\hat{\bf A}  \in \mathbb{R}^{2 \times 2}$ such that $\mathcal{A}{\bf x}^{3} = (\hat{\bf A}{\bf x})^{[3]}.$ Let $\hat{\bf A} = \begin{bmatrix}
a & b \\c & d
\end{bmatrix}.$ Then $\hat{\bf A}{\bf x} = (ax_{1}+bx_{2},cx_{1}+dx_{2})^{T}.$ From $\mathcal{A}{\bf x}^{3} = (\hat{\bf A}{\bf x})^{[3]},$ it follows that
\begin{equation}\label{EQ1}
\begin{bmatrix}
(ax_{1}+bx_{2})^{3}\\
(cx_{1}+dx_{2})^{3}
\end{bmatrix} = \begin{bmatrix}
x_{1}^{3}+x_{2}^{3}\\  -x_{2}^{3}
\end{bmatrix}. 
\end{equation} 
From the first row in Eq. (\ref{EQ1}), after comparing the coefficients, we get $a^{3} =1, b^{3} =1, 3 a b^{2} =0$, and $3 a^{2}b =0$. As there does not exist any real numbers $a$ and $b$ satisfying these conditions, we conclude that the tensor $\mathcal{A}$ cannot be a matrix based tensor.    
\end{example}

The following example illustrates that a matrix-based tensor is not necessarily a row-diagonal tensor.
\begin{example}\rm \label{matrix based but not row-diagonal}
Let $\mathcal{A} = (a_{ijkl}) \in \mathbb{T}(4,2)$ be a matrix based tensor corresponding to a matrix $\hat{\bf A} = \begin{bmatrix}
1 & 1 \\
-1 & 1
\end{bmatrix},$ and an odd exponent $k = 3.$ Then for any ${\bf x} \in \mathbb{R}^{2},$ we have $\mathcal{A}{\bf x}^{3} = ((x_{1}+x_{2})^{3}, (-x_{1}+x_{2})^{3})^{T}.$ It is easy to see that $\mathcal{A}$ is not a row-diagonal tensor as $a_{1112} \neq 0.$ 
\end{example}

\section{Conclusions}\label{Section 6}
This paper explores the finiteness of the solution set of the polynomial complementarity problem, by introducing two new classes of structured tensor tuples, namely non-degenerate tensor tuple and strong non-degenerate tensor tuple. We have provided a sufficient condition for the finiteness of the solution set of the PCP, and  establish an equivalence between a non-degenerate tensor and the finiteness of the solution set of TCP, for the class of matrix based tensors. 
 
\section*{Acknowledgements}
Sonali Sharma  acknowledges the financial support from IIT madras under the institute post-doctoral fellowship (Award No.: F.Acad/IPDF/R12/2025).

\end{document}